\documentclass{amsart}
\usepackage{amssymb}
\usepackage{amsthm}
\usepackage{amsmath,amscd}
\usepackage[mathscr]{euscript}
\usepackage[all]{xy}
\usepackage[utf8]{inputenc}
\usepackage{lmodern}
\usepackage[T1]{fontenc}
\usepackage[textwidth=5.1in, textheight=7.8in, hcentering]{geometry}
\usepackage[pagebackref=true,breaklinks=true,colorlinks]{hyperref}
\theoremstyle{plain}
\newtheorem{theo}{Theorem}[section]
\newtheorem{lemm}[theo]{Lemma}
\newtheorem{prop}[theo]{Proposition}

\theoremstyle{definition}
\newtheorem{defi}[theo]{Definition}

\newtheorem{cons}[theo]{Construction}

\newtheorem{exa}[theo]{Example}

\newtheorem{rem}[theo]{Remark}

\numberwithin{equation}{section}

\newcommand{\op}{^{\mathrm{op}}}
\newcommand{\cat}{\mathscr}

\newcommand{\on}{\mathrm}

\newcommand{\Map}{\on{Map}}

\newcommand{\Cat}{\cat{C}at}

\renewcommand{\S}{\cat{S}}

\newcommand{\CC}{\mathsf{C}}
\newcommand{\DD}{\mathsf{D}}
\newcommand{\EE}{\mathsf{E}}

\newcommand{\NN}{\mathsf{N}}

\newcommand{\WW}{\mathsf{W}}

\newcommand{\id}{\mathrm{id}}

\renewcommand{\L}{\mathfrak{L}}

\newcommand*{\longhookrightarrow}{\ensuremath{\lhook\joinrel\relbar\joinrel\rightarrow}}
\newcommand*{\longhookleftarrow}{\ensuremath{\leftarrow\joinrel\relbar\joinrel\rhook}}
\newcommand{\we}{\stackrel{\simeq}{\longrightarrow}}
\newcommand{\wc}{\stackrel{\simeq}{\longhookrightarrow}}
\newcommand{\lwe}{\stackrel{\simeq}{\longleftarrow}}
\newcommand{\lto}{\longrightarrow}
\newcommand{\lwc}{\stackrel{\simeq}{\longhookleftarrow}}

\title{Brown categories and bicategories}

\author{Geoffroy Horel}

\address{Max Planck Institute for Mathematics,\\
Vivatsgasse 7,\\
53111 Bonn,\\
Germany}

\email{geoffroy.horel@gmail.com}
\thanks{The author was supported by Michael Weiss's Alexander von Humboldt professor grant}
\keywords{}
\subjclass[2010]{18G55,18G30,18D05}

\begin{document}

\begin{abstract}
In a Brown category of cofibrant objects, there is a model for the mapping spaces of the hammock localization in terms of zig-zags of length $2$. In this paper we show how to assemble these spaces into a Segal category that models the infinity-categorical localization of the Brown category.
\end{abstract}

\maketitle

\section{Introduction}
Given a category $\cat{C}$ with a subcategory $w\cat{C}$ of weak equivalences, we can always form the category $\on{Ho}(\cat{C},w\cat{C})$ (or $\on{Ho}(\cat{C})$ if there is no ambiguity). It is uniquely defined up to equivalence of categories by the fact that, for any category $\cat{B}$, the category of functors from $\on{Ho}(\cat{C})$ to $\cat{B}$ is equivalent to the category of functors from $\cat{C}$ to $\cat{B}$ sending the maps of $w\cat{C}$ to isomorphisms in $\cat{B}$. 

It is nowadays well-understood that the homotopy category of $\cat{C}$ is the shadow of a richer object: the $\infty$-categorical localization of $\cat{C}$ at $w\cat{C}$. This is an $\infty$-category with a map from $\cat{C}$ that sends the maps of $w\cat{C}$ to weak equivalences and which is initial (in the $\infty$-categorical sense) with this property. In other words, it is the $\infty$-category that satisfies an $\infty$-categorical version of the characterization of the homotopy category explained in the previous paragraph. One of the most famous model for this $\infty$-categorical localization is the hammock localization of Dwyer and Kan (see~\cite{dwyercalculating}). The output of the hammock localization is a simplicially enriched category. The category of simplicially enriched categories can be equipped with a notion of weak equivalences that make them into a model for $\infty$-categories (\emph{cf.}  \cite{bergnerthree}). Applying the homotopy coherent nerve to the hammock localization, we obtain a quasicategory model of the $\infty$-categorical localization of $\cat{C}$ at $w\cat{C}$ (this fact is proved in~\cite[Proposition 1.2.1.]{hinichdwyer}).

Although extremely useful theoretically, the hammock localization has very complicated mapping spaces built out of arbitrary zig-zags of maps in $\cat{C}$. Fortunately in many cases, it is enough to restrict to much simpler zig-zags. For instance, if $\cat{C}$ is a Brown category of cofibrant objects and $X$ and $Y$ are objects of $\cat{C}$, then, we can model the space of maps from $x$ to $y$ as the nerve of a category whose objects are the objects of the undercategory $\cat{C}^{x\sqcup y/}$ of the form $X\lto Y'\lwc Y$ and whose morphisms are morphisms in the undercategory $\cat{C}^{x\sqcup y/}$ that are sent to weak equivalences by the forgetful functor $\cat{C}^{x\sqcup y/}\to\cat{C}$.
A proof of this fact can be found in~\cite{weisshammock} assuming the existence of functorial factorizations in $\cat{C}$ and in~\cite[Proposition 3.23]{cisinskiinvariance} in full generality. A different proof appears in~\cite[Theorem 4.11.]{lowcocycles}. This fact can also been seen as a corollary of our main result (see Theorem~\ref{theo: main}). A closely related construction appears in work of Jardine under the name \emph{cocycle categories} (see~\cite{jardinecocycle}).

This model of Weiss and Cisinski is very simple but it seems that we have lost one of the key feature of the hammock localization, namely the ability to compose. Indeed, if we have two of these zig-zags $x\lto y'\lwc y$ and $y\lto z'\lwc z$ representing zero simplices in the mapping space from $x$ to $y$ and the mapping space from $y$ to $z$, we can form their composite in the hammock localization. This is given by the concatenation
\[x\lto y'\lwc y\lto z'\lwc z\]
which is sadly not a $0$-simplex in the mapping space from $x$ to $z$. Fortunately, there is another very natural way to compose these two zig-zags. Using the axioms of a Brown category, we can form the diagram
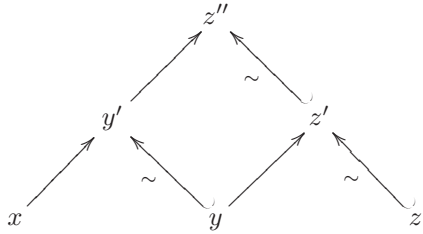
\[\xymatrix{
  &  &z''&  &  \\
  &y'\ar[ur]& &z'\ar@{_{(}->}[ul]^\sim& \\
x\ar[ur]& &y\ar[ur]\ar@{_{(}->}[ul]^\sim& &z\ar@{_{(}->}[ul]^\sim
}
\] 
in which the square is cocartesian. Then we can define the composite of our two zig-zags to be the exterior two edges of the above diagram.

This construction does not quite form a $2$-category since pushouts are only well-defined up to isomorphisms. One way to deal with this issue is to make choices of compositions and then add coherence isomorphisms. That way, we obtain a bicategory assembling all these mapping spaces together. Such a construction can be found in~\cite[Remark 1.2.]{weisshammock} and Cisinski in~\cite[p.524]{cisinskiinvariance}. 

In this paper, we have chosen a different approach that is in our sense more natural. Instead of making choices of compositions, there is a way to package all the possible choices into a single object. The result is a Tamsamani bicategory that we call the Weiss bicategory of $(\cat{C},w\cat{C})$ (see Definition~\ref{defi: weiss bicategory} for a precise construction). A Tamsamani bicategory is a simplicial diagram in categories that satisfies Segal conditions (see Definition~\ref{def: Tamsamani}). Hitting such an object with the nerve functor yields a Segal category. The main result of the present paper asserts that the Segal category that we obtain after applying the nerve functor the Weiss bicategory of $(\cat{C},w\cat{C})$ is a model for the $\infty$-categorical localization of $\cat{C}$ at $w\cat{C}$. More precisely, we prove that it is equivalent to Rezk's relative nerve construction (see Theorem~\ref{theo: main}). The fact that Rezk's relative nerve construction is a model for the $\infty$-categorical localization is a folklore theorem. Since this result seems to be missing from the literature we provide a proof in the appendix (see Theorem~\ref{theo:main appendix}) based on work of To\"en. 

Our main result is valid for categories that are more general than Brown categories of cofibrant objects. We introduce the notion of a partial Brown category in Definition~\ref{defi PBC}. Any Brown category (with functorial cylinders) is a partial Brown category but there are many partial Brown categories that cannot be extended to a Brown category. In short, Brown categories model $\infty$-categories with finite colimits whereas we prove in Proposition~\ref{prop: any rel cat is a PBC}, that partial Brown categories can model all $\infty$-categories.

\subsection*{Notations}

We denote by $\Cat$ the category of small categories. We denote by $\S$ the category of simplicial sets. We call the objects of $\S$ spaces rather than simplicial sets. Given an object $X$ of $\S$, we sometimes use the phrase ``the points of $X$'' to refer to the zero-simplices of $X$.

We denote by $\ast$ the terminal object of both $\S$ and $\Cat$.

We denote by $N:\Cat\to \S$ the nerve functor. We say that a map $f:C\to D$ in $\Cat$ is a \textbf{weak equivalence} if $N(f)$ is a weak equivalence in the Kan-Quillen model structure on $\S$. These maps are the weak equivalences of the Thomason model structure on $\Cat$ constructed in~\cite{thomasoncat}. Note that this notion of weak equivalence is different from the notion of \textbf{equivalence of categories}. An equivalence of category is a weak equivalence but the converse is not true.

We call an object of $\Cat^{\Delta\op}$ a simplicial category. This is a nonstandard terminology but we will never use the notion of simplicially enriched category. The functor $N:\Cat\to\S$ extends to a functor $\Cat^{\Delta\op}\to\S^{\Delta\op}$ that we also denote by $N$.

\section{Partial Brown categories}

Recall that a relative category is a pair $(\cat{C},w\cat{C})$ in which $\cat{C}$ is a category and $w\cat{C}$ is a subcategory containing all the objects. We denote by $\cat{RC}at$ the category of relative categories and weak equivalences preserving maps. Given a category $\cat{C}$, we denote by $\cat{C}$ the relative category $(\cat{C},\on{Ob}(\cat{C}))$ and by $\cat{C}^{\sharp}$ the relative category $(\cat{C},\cat{C})$. There is a functor $N^R_\bullet:\cat{RC}at\to\S^{\Delta\op}$ sending $(\cat{C},w\cat{C})$ to the simplicial space $N_p^R(\cat{C},w\cat{C})_q:=\cat{RC}at([p]\times[q]^\sharp,(\cat{C},w\cat{C}))$. It is proved as the main result of~\cite{barwickrelative} that $\cat{RC}at$ has a model structure whose weak equivalences are the maps that are sent by $N^R_\bullet$ to weak equivalences in the model structure of complete Segal spaces. Moreover the authors of~\cite{barwickrelative} prove that this model category is a model for the $\infty$-category of $\infty$-categories. They also prove in~\cite[Theorem 1.8.]{barwickcharacterization} that the weak equivalences of relative categories are exactly the maps that are sent to Dwyer-Kan equivalences of simplicially enriched categories by the Dwyer-Kan hammock localization.

We now recall the notion of a Brown category. These are relative categories with the additional data of a class of cofibrations. The dual notion was introduced by Brown in~\cite[I.1.]{brownabstract} under the name category of fibrant objects.

\begin{defi}
A \textbf{Brown category} is a category $\cat{C}$ with all finite coproducts and equipped with the data of two subcategories $w\cat{C}$ (whose maps are called the weak equivalences) and $c\cat{C}$ (whose maps are called the cofibrations) satisfying the following axioms.
\begin{enumerate}
\item The weak equivalences satisfy the two-out-of-three property and contain all the isomorphisms.
\item The isomorphisms are cofibrations.
\item The cobase change of a map in $c\cat{C}$ (resp. $w\cat{C}\cap c\cat{C}$) along any map exists and is in $c\cat{C}$ (resp. $w\cat{C}\cap c\cat{C}$).
\item For any $X$ in $\cat{C}$, there exists a factorization of the codiagonal $C\sqcup C\to C$ as the composite of a cofibration $C\sqcup C\to C\otimes I$ followed by a weak equivalence $C\otimes I\to C$.
\item For any object $X$ of $\cat{C}$, the map $\varnothing\to X$ is a cofibration.
\end{enumerate}
\end{defi}

Our main result (Theorem~\ref{theo: main}) does not require the full strength of the axioms of a Brown category. We now introduce the definition of a partial Brown category. This is a structure on a relative category that is weaker than that of a Brown category but still sufficient for our purposes.

\begin{defi}\label{defi PBC}
A \textbf{partial Brown category} (hereafter abbreviated to PBC) is a category $\cat{M}$ with two subcategories $w\cat{M}$ and $c\cat{M}$ whose maps are called respectively the weak equivalences and trivial cofibrations such that the following axioms are satisfied.
\begin{enumerate}
\item Both $w\cat{M}$ and $c\cat{M}$ contain the isomorphisms of $\cat{M}$ and $c\cat{M}$ is contained in $w\cat{M}$.
\item The weak equivalences satisfy the two-out-of-three property.
\item The cobase change of a trivial cofibration along any map exists and is a trivial cofibration.
\item There are three functors $c$, $w$ and $s$ from $w\cat{M}^{[1]}$ to $w\cat{M}^{[1]}$ such that  for each weak equivalence $f$ we have $f=w(f)\circ c(f)$, $w(f)\circ s(f)=\id$ and $c(f)$ and $s(f)$ are in $c\cat{M}^{[1]}$.
\end{enumerate} 
\end{defi}

We shall use the symbol $\we$ to denote weak equivalences and $\wc$ to denote trivial cofibrations.

\begin{rem}
Note that we require functoriality of the factorization. We believe that this axiom could be weakened to simply requiring the existence of a factorization of this type and our main result would remain valid. However this would make the proofs more technical.
\end{rem}

\begin{rem}
One can also introduce the notion dual to that a partial Brown category. We suggest the names partial Brown category of cofibrant objects and partial Brown category of fibrant objects when one needs to distinguish the two notions. All of the results in this paper admit a dual version that holds for partial Brown categories of fibrant objects.
\end{rem}

\begin{defi}
A \textbf{right exact functor} between PBCs $\cat{M}$ and $\cat{N}$ is a functor $f:\cat{M}\to\cat{N}$ sending weak equivalences to weak equivalences, trivial cofibrations to trivial cofibrations and pushout squares of spans of the form $\bullet\longleftarrow\bullet\wc\bullet $ to pushout squares. A \textbf{right exact equivalence} is a right exact functor which is a weak equivalence between the underlying relative categories. 
\end{defi}

We denote by $\cat{PBC}$ the relative category whose objects are small PBCs, morphisms are right exact functors and weak equivalences are right exact equivalences.

\begin{prop}\label{factorization}
Let $\cat{C}$ be a Brown category of cofibrant objects with a functorial cylinder object. Then $\cat{C}$ with the induced notion of weak equivalences and trivial cofibrations is a partial Brown category.
\end{prop}

\begin{proof}
Only the fourth axiom is not obvious. Let us denote by $m\mapsto m\otimes I$ the functorial cylinder object in $\cat{C}$ (the notation does not imply any kind of monoidal structure on $\cat{C}$). Any map $f:m\to n$ can be factored as $m\to (m\otimes I)\sqcup^mn\to n$ where the first map is a cofibration and the second map is a weak equivalence. Moreover, the weak equivalence $(m\otimes I)\sqcup^mn\to n$ has a section given by the obvious map from $n$ to the pushout. This section is a trivial cofibration. The proof of these facts can be found in~\cite[Factorization lemma, p.421]{brownabstract}.
\end{proof}

\begin{exa}
We now list various ways of constructing PBCs:
\begin{itemize}
\item By Proposition~\ref{factorization}, any Brown category (with functorial cylinder) is a PBC. In particular, the category of cofibrant objects of a model category $\cat{M}$ with functorial factorizations (or even a left derivable category as defined in~\cite{cisinskicategories}) is a PBC.
\item The structure of a PBC is stable upon taking a homotopically replete full subcategory (that is a full subcategory on a set of objects which is closed under weak equivalences).
\item If $f:\cat{C}\to\cat{D}$ is an equivalence of categories and one of $\cat{C}$ and $\cat{D}$ is a PBC, then there is a unique PBC structure on the other category that makes the functor $f$ right exact.
\item If $\cat{M}$ and $\cat{N}$ are PBCs, then the coproduct $\cat{M}\sqcup \cat{N}$ has a PBC structure in which a map is a cofibration or weak equivalence if it comes from a cofibration or weak equivalence in $\cat{M}$ or $\cat{N}$. Clearly $\cat{M}\sqcup \cat{N}$ is the coproduct in $\cat{PBC}$. Note that even if $\cat{M}$ and $\cat{N}$ are Brown categories, the coproduct $\cat{M}\sqcup\cat{N}$ fails to have an initial object. In particular, it cannot be given a Brown category structure.
\item  If $\cat{M}$ and $\cat{N}$ are PBCs, then the product $\cat{M}\times \cat{N}$ has a PBC structure in which the cofibrations and weak equivalences are the products of cofibrations or weak equivalences. This makes $\cat{M}\times\cat{N}$ into the product of $\cat{M}$ and $\cat{N}$ in $\cat{PBC}$.
\item If $(C,wC)$ is a small relative category and $\cat{M}$ is a PBC, then the category $\cat{M}^{(C,wC)}$ of relative functors is a PBC if we give it the levelwise weak equivalences and cofibrations.
\end{itemize}
\end{exa}

\begin{prop}\label{prop: any rel cat is a PBC}
Any small relative category $(C,wC)$ is weakly equivalent to the underlying relative category of a partial Brown category.
\end{prop}

\begin{proof}
We start from the category $\cat{M}=\S^{C\op}$ of simplicial presheaves over $C$ and equip it with the injective model structure. Then we can form the left Bousfield localization of $\cat{M}$ with respect to the all the maps between representable presheaves $f_*:C(-,c)\to C(-,c')$ for $f:c\to c'$ a weak equivalence in $C$. This gives a model category $L_w\cat{M}$ in which the cofibrations are the monomorphisms and the fibrant objects are the relative functors from $C\op$ to $\S$ that are also fibrant in the injective model structure on $\S^{C\op}$. Then we can consider the category $Ey$ to be the smallest homotopically replete subcategory of $L_w\cat{M}$ containing the image of the Yoneda's embedding. We equip $Ey$ with its PBC structure inherited from the model structure of $L_w\cat{M}$. The fact the the map $(C,wC)\to Ey$ is a weak equivalence of relative categories is proved in~\cite[3.3.]{barwickpartial}.
\end{proof}

The main reason for our slightly unconventional factorization axiom is for Ken Brown's lemma to remain valid.

\begin{lemm}\label{Ken Brown's lemma}
Let $\cat{M}$ be a PBC, $(\cat{X},w\cat{X})$ be a relative category and $F:w\cat{M}\to \cat{X}$ be a functor. If $F$ sends trivial cofibrations to weak equivalences, then $F$ sends all weak equivalences to weak equivalences.
\end{lemm}

\begin{proof}
Let $w$ be a weak equivalence in $\cat{M}$. Then by the fourth axiom of Definition~\ref{defi PBC}, the map $w$ can be factored as $w=v\circ i$ with $i$ a trivial cofibration and $v$ a weak equivalence which admits a trivial cofibration as a section. The map $F(i)$ is thus a weak equivalence. The map $F(v)$ is also an equivalence because it has a section which is a weak equivalence. Therefore $F(w)$ is a weak equivalence.
\end{proof}

\section{Segal spaces and categories}

The category $\Delta$ is the full subcategory of $\Cat$ spanned by the totally ordered sets $[n]$ with $n\geq 0$. 

Let $\cat{C}$ be a category with finite limits. Given any functor $X:\Delta\op\to\cat{C}$, we can form the limit of the following diagram
\begin{equation}\label{Segal diagram}
X_0\xleftarrow{d_0} X_1\xrightarrow{d_1} X_0\xleftarrow{d_0} X_1\xrightarrow{d_1}
\ldots X_0\xleftarrow{d_0} X_1\xrightarrow{d_1} X_0
\end{equation}
that we denote by 
\[X_1\times_{X_0}X_1\times_{X_0}\ldots\times_{X_0}X_1\]
Let us consider the $n$ maps $[1]\to [n]$ mapping $0$ to $i$ and $1$ to $i+1$ for $0\leq i\leq n-1$ and the $n+1$ maps $[0]\to[n]$ sending $0$ to $i$ for $0\leq i\leq n$. These maps assemble into a single map
\[c_n:X_n\to X_1\times_{X_0}X_1\times_{X_0}\ldots\times_{X_0}X_1\]

If $\cat{C}$ is a model category (in this paper, $\cat{C}$ will always be $\S$ or $\Cat$), we denote by 
\[X_1\times_{X_0}^hX_1\times_{X_0}^hX_1 \times\ldots\times_{X_0}^hX_1\]
the homotopy limit of the diagram~\ref{Segal diagram}. This object is well-defined up to weak equivalence and comes with a map
\[d_n:X_1\times_{X_0}X_1\times_{X_0}\ldots\times_{X_0}X_1\to X_1\times^h_{X_0}X_1\times^h_{X_0}\ldots\times^h_{X_0}X_1\]

\begin{defi}\label{Segal spaces}
A simplicial space $X$ is said to be a \textbf{Segal space} if for each $n\geq 2$, the map 
\[s_n:=d_n\circ c_n:X_n\to X_1\times^h_{X_0}X_1\times^h_{X_0}\ldots\times^h_{X_0}X_1\]
is a weak equivalence.

A simplicial space $X$ is said to be a Segal category if it is a Segal space and $X_0$ is a discrete simplicial set.
\end{defi}

\begin{rem}
The notion of Segal space was introduced in~\cite{rezkmodel}. However, the reader should be aware that our definition is slightly different than the one in~\cite{rezkmodel}. More precisely, a Segal space in Rezk's sense is a Segal space in the sense of~\ref{Segal spaces} which is moreover fibrant in the Reedy model structure of $\S^{\Delta\op}$.
\end{rem}

The maps $s_n$ appearing in the above definition will be called the $n$-th Segal map. Note that if $X_0$ is discrete, the map $d_n$ is always a weak equivalence since products in $\S$ are automatically derived products. Therefore, the map $s_n$ is a weak equivalence if and only if the map $c_n$ is a weak equivalence.

Let $X$ be a simplicial space (resp. a simplicial category) with $X_0$ a discrete space (resp. a discrete category). There is a map $p_n:X_n\to X_0^{n+1}$ whose $i$-th factor is the map $X_n\to X_0$ induced by the map $[0]\to[n]$ sending $0$ to $i$. Given $(x_0,\ldots,x_n)\in X_0^{n+1}$, we denote by $X(x_0,\ldots,x_n)$ the fiber of $p_n$ over $(x_0,\ldots,x_n)$. Note that in this situation, this is also a homotopy fiber.

We denote by $\cat{CSS}$ the model structure of complete Segal spaces constructed in~\cite{rezkmodel}. This is a model structure on $\S^{\Delta\op}$ which is a model for the $\infty$-category of $\infty$-categories. We use the phrase \textbf{Rezk equivalences} for the weak equivalences in $\cat{CSS}$. The following lemma gives us a criterion for a map to be a Rezk equivalence.

\begin{lemm}\label{lemm: criterion Rezk equivalence}
Let $f:X\to Y$ be a morphism between Segal spaces. Then $f$ is a Rezk equivalence if it satisfies the following two conditions.
\begin{itemize}
\item The square
\[
\xymatrix{
X_1\ar[d]_{(d_0,d_1)}\ar[r]^{f_1}& Y_1\ar[d]^{(d_0,d_1)}\\
X_0\times X_0\ar[r]_{f_0\times f_0}& Y_0\times Y_0
}
\]

is homotopy cartesian.

\item The induced map $\pi_0(X_0)\to\pi_0(Y_0)$ is surjective.
\end{itemize}
\end{lemm}

\begin{proof}
Let $f$ be a map satisfying the above two conditions. Let $X\mapsto X^R$ be a fibrant replacement functor in the Reedy model structure on $\S^{\Delta\op}$. The map $X\to X^R$ and $Y\to Y^R$ are weak equivalences in $\cat{CSS}$ since the latter model category is a left Bousfield localization of $\S^{\Delta\op}$ with its Reedy model structure. Therefore, the map $f$ is a weak equivalence if and only if the map $f^R:X^R\to Y^R$ is a weak equivalence. Moreover, the map $f^R$ also satisfies the conditions of the lemma since those are invariant under levelwise weak equivalences. We may therefore assume that $X$ and $Y$ are Reedy fibrant. In that case, according to~\cite[Theorem 7.7.]{rezkmodel}, we see that $f$ is a weak equivalence if and only if it is a Dwyer-Kan equivalence. According to~\cite[Proposition 2.15]{horelmodel}, the map $f$ is a Dwyer-Kan equivalence if and only if it satisfies the first condition of the lemma and the induced map $\pi_0(X_0)/\sim\to\pi_0(Y_0)/\sim$ is a surjection where $\pi_0(X_0)/\sim$ is a certain functorial quotient of $\pi_0(X_0)$. But it is easy to see that if $\pi_0(X_0)\to\pi_0(Y_0)$ is surjective, then the map $\pi_0(X_0)/\sim\to\pi_0(Y_0)/\sim$ must be surjective as well .
\end{proof}

\begin{defi}\label{def: Tamsamani}
A simplicial category $X$ is said to be a \textbf{Tamsamani bicategory} if $X_0$ is a discrete category and the maps $c_n:X_n\to X_1\times_{X_0}X_1\times_{X_0}\ldots\times_{X_0}X_1$ are equivalences of categories for all $n\geq 2$.
\end{defi}

\begin{rem}\label{rem: bicategory}
Tamsamani bicategories should be thought of as a mild generalization of bicategories. In fact, Lack and Paoli construct in~\cite{lack2} a fully faithful $2$-nerve from bicategories to Tamsamani bicategories. Note that since equivalences of categories are in particular weak equivalences, the simplicial space obtained by applying the nerve levelwise to a Tamsamani bicategory is a Segal category.
\end{rem}

\section{Fibrillations and Quillen's theorem B}

Computing homotopy pullbacks in a right proper model category usually involves replacing one of the maps by a fibration. It has been observed by Rezk~\cite{rezkfibrations} that a weaker kind of fibrations is good enough. Those arrows are called sharp maps by Rezk, fibrillations by Barwick and Kan in~\cite{barwickquillen} and weak fibrations by Cisinski in~\cite{cisinskiinvariance}. We use Barwick and Kan's terminology. Let $\cat{X}$ be a right proper model category. The class of \textbf{fibrillations} is defined to be the largest class of maps in $\cat{X}$ with the property that a pullback of a weak equivalence along a fibrillation is a weak equivalence and the pullback of a fibrillation is a fibrillation. Note that any fibration is a fibrillation. Moreover, we have the following proposition.

\begin{prop}{\cite[Proposition 2.7.]{rezkfibrations}}\label{fibrillations and hpb}
Let $\cat{X}$ be a right proper model category and let
\[\xymatrix{
a\ar[d]\ar[r]&b\ar[d]^p\\
a'\ar[r]&b'
}
\]
be a square in which $p$ is a fibrillation. Then the square is homotopy cartesian if and only if the map $a\to a'\times_{b'}b$ is a weak equivalence. $\hfill\square$
\end{prop}

We will be mainly interested in the fibrillations in $\Cat$ with the Thomason model structure. An elaboration of Quillen's theorem B due to Barwick and Kan gives a very efficient way to produce fibrillations in $\Cat$.

Given a functor $F:\cat{X}\op\to\Cat$, we define its Grothendieck construction $\on{Gr}(F)$. The objects of this category are pairs $(X,a)$ with $X$ an object of $\cat{X}$ and $a$ an object of $F(X)$. The set of morphisms from one such object $(X,a)$ to another object $(Y,b)$ is the set of pairs $(f,u)$ with $f:X\to Y$ a map in $\cat{X}$ and $u:a\to F(f)b$ a map in $F(X)$. The composition is defined in a straightforward way. The category $\on{Gr}(F)$ comes equipped with a map $\on{Gr}(F)\to\cat{X}$. 
 
Following Barwick and Kan, we say that a functor $F$ from a small category $\cat{X}$ to $\Cat$ has property $Q$ if it sends all maps in $\cat{X}$ to weak equivalences.

\begin{prop}\label{Theorem B}
Let $F:\cat{X}\op\to\Cat$ be a functor having property $Q$. Let $\on{Gr}(F)$ be the Grothendieck construction of $F$. The induced map $\on{Gr}(F)\to \cat{X}$ is a fibrillation in $\Cat$. $\hfill\square$
\end{prop}

\begin{proof}
The proof is essentially done in~\cite[Lemma 9.7.]{barwickquillen}. Our definition of the Grothendieck construction differs from the one in Barwick and Kan. They work with covariant functors as opposed to contravariant functors in our case. However, denoting by $\on{Gr}'$ the Grothendieck construction used by Barwick and Kan, it is easy to see that for $\cat{X}\op\to\Cat$ a functor, the map $\on{Gr}(F)\to\cat{X}$ is isomorphic to the opposite of the map $\on{Gr}'(F\op)\to\cat{X}\op$ where $F\op$ denotes the functor $\cat{X}\op\to\Cat$ sending $X$ to $F(X)\op$. 

Now, we observe that the functor $(-)\op:\Cat\to\Cat$ preserves weak equivalences. It follows that $F\op$ has property $Q$. Therefore, by~\cite[Lemma 9.7.]{barwickquillen}, the map $$\on{Gr}'(F\op)\to\cat{X}\op$$ is a fibrillation. Now the class of fibrillation is preserved by the functor $(-)\op$ since it preserves and reflects pullback squares and weak equivalences.
\end{proof}

\section{A Segal space associated to a PBC}

For any integer $n$, we construct a category $T_n$. This is the full subcategory of $[n]\op\times[n]$ spanned by objects $(p,q)$ with $p\leq q$.

\begin{cons}\label{cons CM}
Let $\cat{M}$ be a PBC. We define a category $\CC_n(\cat{M})$. The objects of this category are the functors $m:T_n\to \cat{M}$ satisfying the following conditions:
\begin{enumerate}
\item Maps of the form $(g,id)$ for $g$ a map in $[n]\op$ are sent to trivial cofibrations.
\item For any $p<q<n$, the square
\[\xymatrix{
m_{p+1,q}\ar[d]\ar[r]& m_{p,q+1}\ar[d]\\
m_{p+1,q+1}\ar[r] &m_{p,q+1}
 }
\]
is a pushout square. 
\end{enumerate}
The morphisms in $\CC_n(\cat{M})$ are the natural transformations that are objectwise in $w\cat{M}$.
\end{cons}

For example an object in $\CC_2(\cat{M})$ is a diagram of the form 
\[\xymatrix{
  &  &m_{02}&  &  \\
  &m_{01}\ar[ur]& &m_{12}\ar@{_{(}->}[ul]^\simeq& \\
m_{00}\ar[ur]& &m_{11}\ar[ur]\ar@{_{(}->}[ul]^\simeq& &m_{22}\ar@{_{(}->}[ul]^\simeq
}
\]
in which the square is a pushout square. 

The assignment $[n]\to T_n$ is a cosimplicial category. This makes $\CC_n(\cat{M})$ into a simplicial category. For instance the three face maps from $\CC_2(\cat{M})$ to $\CC_1(\cat{M})$ send the object above respectively to 
$m_{00}\lto m_{01}\lwc m_{11}$, $m_{00}\lto m_{02}\lwc m_{22}$ and $m_{11}\lto m_{12}\lwc m_{22}$. The degeneracies are obtained by inserting identities. We define $C(\cat{M})$ to be the simplicial space obtained from $\CC(\cat{M})$ by applying the nerve objectwise. 

Let $N^R(\cat{M})$ be Rezk classifying diagram of $\cat{M}$. This is a simplicial space whose space of $n$-simplices $N^R_n(\cat{M})$ is the nerve of the category $\NN^R_n(\cat{M})$ whose objects are functors $[n]\to \cat{M}$ and morphisms are natural transformations which are objectwise weak equivalences. 

We now construct a map $\NN^R(\cat{M})\to \CC(\cat{M})$. There is an inclusion $[k]\to T_k$ sending $p$ to $(0,p)$. This induces a restriction map $\CC_k(\cat{M})\to \NN^R_k(\cat{M})$. These restriction maps are not compatible with the simplicial structure. However, for each $k$, the restriction map $\CC_k(\cat{M})\to \NN^R_k(\cat{M})$ has a left adjoint. For instance, in degree $2$, this left adjoint sends $c_0\lto c_1\lto c_2$ to
\[\xymatrix{
  &  &c_{2}&  &  \\
  &c_{1}\ar[ur]& &c_{2}\ar[ul]^\id& \\
c_{0}\ar[ur]& &c_{1}\ar[ur]\ar[ul]^\id& &c_{2}\ar[ul]^\id
}
\] 

These left adjoints are compatible with the simplicial structure. More precisely, they induce a natural transformation $\NN^R(\cat{M})\to \CC(\cat{M})$ that is functorial in $\cat{M}$. Hence we have the following proposition.

\begin{prop}\label{CM is correct}
The map $\NN^R(\cat{M})\to\CC(\cat{M})$ is a levelwise weak equivalence. $\hfill\square$
\end{prop}

Another important feature of the simplicial space $C(\cat{M})$ is the following proposition:

\begin{prop}\label{prop: CM is a Segal space}
The simplicial space $C(\cat{M})$ is a Segal space.
\end{prop}

\begin{proof}
We write $C$ for $C(\cat{M})$ and $\CC$ for $\CC(\cat{M})$. We have the map
\[c_n:\CC_n\to \CC_1\times_{\CC_0}\CC_1\times_{\CC_0}\ldots\times_{\CC_0}\CC_1\]
Since pushouts are determined up to unique isomorphism, this map is an equivalence of categories. The functor $N$ preserves homotopy pullbacks, therefore it suffices to prove that the map
\[\CC_1\times_{\CC_0}\CC_1\times_{\CC_0}\ldots\times_{\CC_0}\CC_1\to
\CC_1\times^h_{\CC_0}\CC_1\times^h_{\CC_0}\ldots\times^h_{\CC_0}\CC_1\]
is a weak equivalence. We denote the source by $\DD_n$ and the target by $\EE_n$. We proceed by induction on $n$. For $n=1$, the map $\DD_1\to\EE_1$ is a weak equivalence. Now, we claim that the map $d_0:\CC_1\to\CC_0$ is a fibrillation. Assuming this is true for the moment, we finish the proof. We assume that $\DD_n\to\EE_n$ is a weak equivalence. By Proposition~\ref{fibrillations and hpb}, this implies that 
\[\DD_{n+1}:=\DD_n\times_{\CC_0}\CC_1\to\EE_n\times_{\CC_0}\CC_1\]
is a weak equivalence and that $\EE_n\times_{\CC_0}\CC_1\to \EE_n\times^h_{\CC_0}\CC_1=\EE_{n+1}$ is a weak equivalence. 

Now we prove that $d_0:\CC_1\to\CC_0$ is a fibrillation. We construct a functor $P:\CC_0\op\to\Cat$ sending $m$ to the category whose objects are zig-zags of the form $m\lto \bullet\lwc\bullet$ and morphisms are objectwise weak equivalences inducing the identity of $m$ on the leftmost term. The category $\CC_1$ can be identified with the Grothendieck construction of $P$. Thus by Proposition~\ref{Theorem B}, it suffices to prove that $P$ has property $Q$. By Lemma~\ref{Ken Brown's lemma}, it is enough to prove that $P$ sends trivial cofibrations to weak equivalences. Let $u:m\to n$ be a trivial cofibration, then the map $P(u):P(n)\to P(m)$ has a left adjoint sending $m\lto a\lwc b$ to $n\lto n\sqcup^ma\lwc b$.
\end{proof}

\section{A fibrancy property}

In this section, we prove that $C(\cat{M})=N\CC(\cat{M})$ is a Segal space that satisfies a weak form of Reedy fibrancy. We first make a definition.

\begin{defi}
A map $f:C\to D$ between categories is said to be a \textbf{quasifibration} if for any object $d\in D$, the square
\[
\xymatrix{
C\times_D\ast\ar[r]\ar[d]&C\ar[d]^f\\
\ast\ar[r]_d&D
}
\]
is homotopy cartesian in $\Cat$.
\end{defi}

This notion is entirely analogous to the notion of quasifibration in $\S$ (defined for instance in~\cite[Definition 1.1.]{doldquasifibration}). More precisely, since $N$ preserves homotopy cartesian squares, we see that a map $f$ in $\Cat$ is a quasifibration if and only if $N(f)$ is a quasifibration in $\S$.

\begin{prop}\label{CM is fibrillant}
For each $n$, the map
\[\CC_n\to \CC_0^{n+1}\]
is a quasifibration.
\end{prop}

\begin{proof}
We denote by $\DD_n$ the category $\CC_1\times_{\CC_0}\times \CC_1\ldots\times_{\CC_0}\CC_1$. This is the category whose objects are zig-zags in $\cat{M}$ of the form
\[m_0\lto m'_0\lwc m_1\lto\ldots\lwc m_{n-1}\lto m'_n\lwc m_n\]
and morphisms are natural weak equivalences between them. The map of the proposition factors as $\CC_n\to \DD_n\to \CC_0^{n+1}$. We denote by $\EE_n$ the category whose objects are zig-zags in $\cat{M}$ of the form
\[m_0\lto m'_0\lwe m_1\lto\ldots\lwe m_{n-1}\lto m'_n\lwe m_n\]
and morphisms are natural weak equivalences between them. Note that this notation conflicts with the one in the proof of Proposition~\ref{prop: CM is a Segal space}. There is an inclusion $\DD_n\to\EE_n$ and the map $\CC_n\to \CC_0^{n+1}$ factors as $\CC_n\to\DD_n\to\EE_n\to\CC_0^{n+1}$.

(1) We first claim that the map $\CC_n\to \DD_n$ is a weak equivalence. In fact, as in the proof of Proposition~\ref{prop: CM is a Segal space}, it is an equivalence of categories since pushouts are determined up to a unique isomorphism.

(2) We claim that the map $\DD_n\to \EE_n$ is a weak equivalence. We do the case $n=1$, to keep the notations simple.  The general case is a straightforward generalization. We denote by $\alpha:\DD_1\to\EE_1$ the inclusion. We construct a map $\beta:\EE_1\to \DD_1$. It sends $m\xrightarrow{f} a\xleftarrow{k} n$ to $m\stackrel{s(k)\circ f}{\longrightarrow} a'\stackrel{c(k)}{\longleftarrow} n$ where $s$ and $c$ are the functor appearing in the factorization axiom (see axiom 4 of Definition~\ref{defi PBC}). Note that we have the following commutative diagram:
\[
\xymatrix{
m\ar[r]^f&a&n\ar[l]_{k}\\
m\ar[r]_{s(k)\circ f}  \ar[u]^{\id}           &a'\ar[u]^{w(k)}&n\ar[l]^{c(k)}\ar[u]_{\id}
}
\]
This immediately implies that there are natural transformation $\beta\circ\alpha\to\id_{\DD_1}$ and $\alpha\circ\beta\to\id_{\EE_1}$.

(3) Now we prove that the map $\EE_n\to \CC_0^{n+1}$ is a fibrillation. For $(m_0,\ldots,m_n)$ a sequence of objects of $\cat{M}$, we denote by $\EE(m_0,m_1,\ldots,m_n)$ the fiber of $\EE_n$ over $(m_0,\ldots,m_n)\in\CC_0^{n+1}$. The assignment $(m_0,\ldots,m_n)\mapsto \EE(m_0,\ldots,m_n)$ extends to a functor $P:(\CC_0^{n+1})\op\to\Cat$ and the category $\EE$ is isomorphic to the Grothendieck construction of $P$. Hence, by Proposition~\ref{Theorem B} it suffices to prove that $P$ has property $Q$. By Lemma~\ref{Ken Brown's lemma}, it suffices to prove that $P$ sends products of $n+1$ trivial cofibrations to weak equivalences. Again, we restrict ourselves to the case $n=1$ to simplify the notations. We want to prove that the functor $P:(\CC_0^{2})\op\to\Cat$ sends products of trivial cofibrations to weak equivalence. We can restrict to maps of the form $(u,\id_m)$ and $(\id_m,u)$ where $u$ is a trivial cofibration of $\cat{M}$ since any product of trivial cofibrations is a composite of maps of this form. 

But now, we claim that if $u:a\to b$ is a trivial cofibration the map
\[P(u):\EE(b,m)\to \EE(a,m)\]
has a left adjoint sending $a\lto m'\lwe m$ to $b\lto b\sqcup^am'\lwe m$. The other case is similar.

(4) Hence, by the previous three paragraphs and by~\ref{fibrillations and hpb}, it suffices to prove that the map $\CC(m_0,\ldots,m_n)\to\EE(m_0,\ldots,m_n)$ is a weak equivalence for any object $(m_0,\ldots,m_n)$ in $\CC_0^{n+1}$ (where $\EE(m_0,\ldots,m_n)$ denotes the fiber of the map $\EE_n\to\CC_0^{n+1}$ over $(m_0,\ldots,m_n)$). We can factor this map through $\DD(m_0,\ldots,m_n)$. The map $\CC(m_0,\ldots,m_n)\to\DD(m_0,\ldots,m_n)$ is an equivalence of categories by uniqueness up to a unique isomorphism of pushouts exactly as in (1). Similarly, an argument analogous to (2) implies that $\DD(m_0,\ldots,m_n)\to\EE(m_0,\ldots,m_n)$ is a weak equivalence.
\end{proof}

\section{A Segal category}

Recall that $\cat{CSS}$ denotes the category of simplicial spaces with Rezk's complete Segal space model structure (constructed in~\cite{rezkmodel}). We denote by $\cat{S}e\Cat$ the category of Segal precategories (that is of functors $\Delta\op\to\S$ which are discrete in degree $0$) with Bergner's injective model structure (constructed in~\cite[Theorem 5.1.]{bergnerthree}). 
By~\cite[Theorem 6.3.]{bergnerthree}, we have a Quillen equivalence:
\[i:\cat{S}e\Cat\leftrightarrows\cat{CSS}:\delta\]
where $i$ is the inclusion and $\delta$ its right adjoint.

The functor $\delta$ is constructed explicitly in~\cite[Section 6]{bergnerthree} under the name $R$. There is a similar functor also denoted $\delta$ sending a simplicial category $X$ to the pullback
\[\xymatrix{
\delta(X)\ar[d]\ar[r]& X\ar[d]\\
\on{cosk}_0\on{Ob}(X_0)\ar[r]& \on{cosk}_0X_0
}
\]
where $\on{cosk}_0$ denotes the zero-coskeleton functor. We have an obvious isomorphism $N(\delta X)\cong \delta N(X)$.

\begin{defi}
We say that a simplicial space $X$ is \textbf{quasifibrant} if the canonical map $X\to \on{cosk}_0X_0$ is levelwise a quasifibration.
\end{defi}

We observe that if $X$ is quasifibrant and is a Segal space, then $\delta(X)$ is a Segal category. The following proposition implies that for quasifibrant simplicial spaces, the functor $\delta$ coincides with its derived functor.

\begin{prop}\label{derived delta}
Let $f:X\to Y$ be a map between quasifibrant Segal spaces. Then, $f$ is a Rezk equivalence if and only if $\delta(f)$ is a weak equivalence in $\cat{S}e\Cat$.
\end{prop}

\begin{proof}
(1) First, we claim that for $X$ quasifibrant, the counit map $i\delta(X)\to X$ is a weak equivalence in $\cat{CSS}$. We use Lemma~\ref{lemm: criterion Rezk equivalence}. We have a cartesian square
\[\xymatrix{
i\delta(X)_1\ar[r]\ar[d]&X_1\ar[d]\\
i\delta(X)_0\times i\delta(X)_0\ar[r]&X_0\times X_0
}
\]
By assumption, the map $X_1\to X_0\times X_0$ is a quasifibration. Hence the square is homotopy cartesian by definition of a quasifibration. The second hypothesis of Lemma~\ref{lemm: criterion Rezk equivalence} is satisfied because the map $i\delta(X)_0\to X_0$ is a bijection on the zero-simplices. 

(2) Now, we prove the proposition. We have a commutative diagram in $\S^{\Delta\op}$:
\[\xymatrix{
i\delta(X)\ar[d]_{i\delta (f)}\ar[r]&X\ar[d]^f\\
i\delta(Y)\ar[r]&Y
}
\]
By the previous paragraph and the two-out-of-three property, we see that $f$ is a weak equivalence if and only if $i\delta(f)$ is a weak equivalence. It is shown in the proof of~\cite[Theorem 6.3.]{bergnerthree} that the functor $i$ reflects weak equivalences. Thus $f$ is a weak equivalence if and only if $\delta(f)$ is a weak equivalence.
\end{proof}

We are now ready to introduce the main construction of this paper and prove our main theorem.
 
\begin{defi}\label{defi: weiss bicategory}
We define the \textbf{Weiss bicategory} of $\cat{M}$ to be the functor $\WW(M):=\delta \CC(\cat{M})$. We also define $W(\cat{M}):=N\WW(\cat{M})$.
\end{defi}

\begin{rem}
Note that the Weiss bicategory $\WW(\cat{M})$ is not a bicategory but merely a Tamsamani bicategory (see Remark~\ref{rem: bicategory} for an explanation of the relationship between these two concepts). Its name comes from the fact that this construction is very similar to the bicategory constructed in~\cite[Remark 1.2.]{weisshammock}. 
\end{rem}

By Proposition~\ref{CM is correct}, the functor $C$ is levelwise weakly equivalent to to $N^R$ and the latter functor sends weak equivalences of relative categories to weak equivalences. It follows that $C$ sends right exact equivalences to weak equivalences. By Proposition~\ref{derived delta} and Proposition~\ref{CM is fibrillant}, it follows that the functors $\WW$ and $W$ carry right exact equivalences to weak equivalences. Moreover, the following theorem together with Theorem~\ref{theo:main appendix} insures that the Segal category $W(\cat{M})$ is a model for the $\infty$-categorical localization of $\cat{M}$.

\begin{theo}\label{theo: main}
The Segal category $W(\cat{M})$ is connected to $N^R(\cat{M})$ by a functorial zig-zag of Rezk equivalences.
\end{theo}

\begin{proof}
The counit $iW(\cat{M})\to C(\cat{M})$ is a weak equivalence by the first paragraph of the proof of Proposition~\ref{derived delta}. Moreover, there is a functorial weak equivalence $N^R(\cat{M})\to C(\cat{M})$ by Proposition~\ref{CM is correct}.
\end{proof}

\appendix

\section{Rezk nerve as $\infty$-localization}

The purpose of this appendix is to give a short proof of the fact that Rezk's relative nerve construction  models the $\infty$-categorical localization. This proof is sketched by Christopher Schommer-Pries on Mathoverflow (see~\cite{schommerlocalization}). We refer the reader to the second section for background material and notations about relative categories.

We denote by $\cat{QC}at$ the category of simplicial sets with the Joyal model structure (see~\cite[Theorem 2.2.5.1.]{lurietopos}). We denote by $\cat{CSS}$ the category of simplicial spaces with the Rezk model structure (constructed in~\cite{rezkmodel}). We denote respectively by $\cat{QC}at^f$ and $\cat{CSS}^f$ the full subcategories spanned by the quasicategories and the Reedy fibrant complete Segal spaces. For $i\geq 0$, we denote by $F(i)$ the nerve of $[i]$ seen as a levelwise discrete simplicial space.

For $(\cat{C},w\cat{C})$ a (not necessarily small) relative category, we denote by $\L(\cat{C},w\cat{C})$ or $\L\cat{C}$ if there is no ambiguity the quasicategory obtained by taking the homotopy coherent nerve of a fibrant replacement (in the Bergner model structure on simplicially enriched categories) of the hammock localization of $\cat{C}$. We denote by $L:\cat{RC}at\to\cat{QC}at$ the restriction of $\L$ to small relative categories.

It is proved in~\cite[Proposition 1.2.1.]{hinichdwyer} that for a relative category $(\cat{C},w\cat{C})$, the quasicategory $L(\cat{C},w\cat{C})$ is a model for the $\infty$-categorical localization in the sense that it represents a functor weakly equivalent to the functor $\cat{QC}at^f\to\S$ sending $X$ to 
\[\Map(N\cat{C},X)\times^h_{\Map(Nw\cat{C},X)}\Map(Nw\cat{C},\iota X)\]
where $\iota X$ is the maximal Kan complex contained in $X$.

Let $p_1^*:\cat{QC}at\to\cat{CSS}$ be the functor sending a simplicial set to the same simplicial set seen as a levelwise discrete simplicial space. It is proved in~\cite{joyalquasi} that the functor $p_1^*$ is a left Quillen equivalence. Using Hinich result, we can thus deduce easily that the functor $\cat{CSS}^f\to\S$ sending $X$ to $\Map(p_1^*L(\cat{C},w\cat{C}),X)$ coincides up to natural weak equivalences with the functor
\[X\mapsto\Map(p_1^*N\cat{C},X)\times^h_{\Map(p_1^*Nw\cat{C},X)}\Map(p_1^*Nw\cat{C},cX_0)\]
where $cX_0$ denotes the constant simplicial space on the space $X_0$. Our goal is to show that $p_1^*L(\cat{C},w\cat{C})$ is weakly equivalent to  $N^R(\cat{C},w\cat{C})$.

Our proof is based on a theorem of To\"en. We denote by $\tau$ the map from $\cat{CSS}$ to itself induced by precomposition with the functor $(-)\op:\Delta\to\Delta$.

\begin{theo}[To\"en]\label{p:Toen}
Let $\cat{C}$ be a relative category and $f,g:\L\cat{C}\to\L\cat{CSS}$ be two categorical equivalences (i.e. weak equivalences in Joyal's model structure). Then either $f$ is homotopic to $g$ or $f$ is homotopic to $\L(\tau)\circ g$.
\end{theo}

\begin{proof}
The map $\L(\tau)$ induces a map $\mathbb{Z}/2\to\cat{QC}at(\L\cat{CSS},\L\cat{CSS})$. According to~\cite[Théorème 6.3.]{toenaxiomatisation}, it induces a weak equivalence from $\mathbb{Z}/2$ to the group of homotopy automorphisms of $\L\cat{CSS}$. Let $h$ be a map $\L\cat{CSS}\to\L\cat{C}$ which is a homotopy inverse to $f$. Then $g\circ h$ is a homotopy automorphism of $\L\cat{CSS}$ and hence is homotopic to $\L(\tau)$ or to $\L(\id)$. Therefore $f$ is either homotopic to $g$ or to $\L(\tau)\circ g$.
\end{proof}

We can now prove the main result of this appendix.

\begin{theo}\label{theo:main appendix}
Let $(\cat{C},w\cat{C})$ be a relative category. Then, in the model category $\cat{CSS}$, the object $N^R(\cat{C},w\cat{C})$ is weakly equivalent to $p_1^*L(\cat{C},w\cat{C})$.
\end{theo}

\begin{proof}
We first observe that $p_1^*L$ and $N^R$ are two equivalences of relative categories from $\cat{RC}at$ to $\cat{CSS}$. For $N^R$, this follows from~\cite[Lemma 5.4. and Theorem 6.1.]{barwickrelative}. For $p_1^*L$ this follows from the fact that $p_1^*$ is an equivalence by the main theorem of~\cite{joyalquasi} and that $L$ is an equivalence as the composite of two equivalences (the hammock localization is an equivalence by~\cite[Proposition 3.1. and section 2.5.]{barwickcharacterization} and the homotopy coherent nerve is an equivalence by~\cite[Theorem 2.2.5.1.]{lurietopos}). Hence, we have two relative equivalences $p_1^*L$ and $N^R$ from $\cat{RC}at$ to $\cat{CSS}$. If we can prove that $\L(p_1^*L)$ is homotopic to $\L(N^R)$, this will imply that for any relative category $N^R(\cat{C},w\cat{C})$ is connected by a zig-zag of weak equivalences to $p_1^*L(\cat{C},w\cat{C})$ and conclude the proof.

According to Proposition~\ref{p:Toen}, it suffices to prove that $\L (p_1^*L)$ is not homotopic to $\L(\tau\circ N^R)$. Let us consider the map $(d^0,d^1):[0]\sqcup[0]\to[1]$. We observe that the functor $p_1^*L$ sends this map to
\[(d^0,d^1):F(0)\sqcup F(0)\to F(1)\]
and $\tau\circ N^R$ sends it to
\[(d^1,d^0):F(0)\sqcup F(0)\to F(1)\]
If the two functors were homotopic, by passing to the homotopy category of $\cat{CSS}$, we would find a natural isomorphism between the following two functors from $\on{Ho}\cat{CSS}$ to equalizer diagrams of sets
\[X\mapsto ((d_0,d_1):\pi_0(X_1)\rightrightarrows \pi_0(X_0)),\;\; \mathrm{ and }\;\; X\mapsto ((d_1,d_0):\pi_0(X_1)\rightrightarrows \pi_0(X_0))\]

In particular, taking $X$ to be the nerve of the category freely generated by the oriented graph $\bullet\leftarrow \bullet\rightarrow\bullet$, we would find an isomorphism between this graph and the graph $\bullet\rightarrow\bullet\leftarrow\bullet$ which is a contradiction.
\end{proof}

\bibliographystyle{alpha}
\bibliography{biblio}

\end{document}